\documentclass{amsart}
\usepackage{amsfonts,amssymb,amscd,amsmath,enumerate,verbatim,calc}
\usepackage[all]{xy}

\newcommand{\CM}{Cohen-Macaulay}

\newcommand{\n}{\mathfrak{n} }
\newcommand{\m}{\mathfrak{m} }

\newcommand{\ZZ}{\mathbb{Z} }

\newcommand{\rt}{\rightarrow}

\newcommand{\ov}{\overline}

\newcommand{\bu}{\mathbf{u}}

\newcommand{\divS}{\operatorname{div}}

\newcommand{\Ass}{\operatorname{Ass}}

\newcommand{\grade}{\operatorname{grade}}

\newcommand{\height}{\operatorname{height}}

\newcommand{\cdim}{\operatorname{cdim}}

\theoremstyle{plain}

\newtheorem{theorem}{Theorem}[section]

\newtheorem{lemma}[theorem]{Lemma}
\newtheorem{proposition}[theorem]{Proposition}

\theoremstyle{definition}

\newtheorem{s}[theorem]{}
\newtheorem{remark}[theorem]{Remark}
\newtheorem{example}[theorem]{Example}

\theoremstyle{remark}

\begin{document}

\title{Asymptotic primes and the Chow group }
 \author{Tony J. Puthenpurakal}
\date{\today}
\address{Department of Mathematics, Indian Institute of Technology Bombay, Powai, Mumbai 400 076, India}
\email{tputhen@math.iitb.ac.in}
\begin{abstract}
In this paper we present an unexpected connection between the theory of asymptotic prime divisors and Chow groups. 
As an application we show that the Chow group $A_1(R)$ is a torsion group when $R$ is any graded ring such  that we have an inclusion
of graded rings $T \subseteq R \subseteq S$ where 
\begin{enumerate}
 \item 
 $S = k[X_1, X_2, Y]/(X_1^m + X_2^m + Y^n)$ where $k$ is algebraically field, \\ $(m,n) = 1$, $(mn)^{-1} \in k$, 
$m, n \geq 2$. We consider $S$ graded with $\deg X_1  = \deg X_2 = n$ and $\deg Y = m$.
\item
$T = k[X_1, X_2]$ where $\deg X_1 = \deg X_2 = n$.
\end{enumerate}
We also consider higher dimensional analogues of this result.
\end{abstract}

\maketitle

\section{introduction}
Let 
$A$ be a Noetherian ring and let $M$ be a finitely generated $A$-module. Let $\Ass_A (M)$ denote the set of associate primes of $M$.
A basic result in commutative algebra is that $\Ass_A(M)$ is a finite set. Set
\[
 \Ass^{(i)}_A(M) = \{ P \mid P \in \Ass_A(M) \ \text{and} \ \height P = i \}.
\]

The first result in the theory of asymptotic primes is due to Broadmann \cite{Brod}. It states  that if $I$ is an ideal in $A$ then 
\[
 \bigcup_{n \geq 1} \Ass_A M/I^n M   \quad \text{is a finite set}.
\]
This result has a lot of applications. Historically the set $\Ass_A A/\ov{I^n}$ was considered where $\ov{I^n}$ is the integral 
closure of $I^n$. Ratliff, see \cite{Rat}, showed that $\Ass_A A/\ov{I^n} \subseteq \Ass_A A/\ov{I^{n+1}}$ for all $n \geq 1$ and that
\[
 \bigcup_{n \geq 1} \Ass_A A/\ov{I^n} \quad \text{is a finite set}.
\]
A good reference for these results is McAdam \cite{Mc}.

This paper arose due to my investigation of a different class of asymptotic primes which I now describe.
 Let $(A,\m)$ either be a Noetherian local ring  with maximal ideal $\m$ or a standard graded algebra over a field with  its irrelevant graded maximal ideal $\m$. Let  $x \in  \m^s \setminus \m^{s+1}$ be an $A$-regular element (it is homogeneous if $A$ is graded). Set
\[
A_s(x) = \sharp \Ass_A^{(1)} A/(x).
\]
Here $\sharp T$ denotes the number of elements in a set $T$.
Set
\[
A(s) = \{ A_s(x) \mid x \in \m^s \setminus \m^{s+1} \ \text{is $A$-regular} \}.
\]
Our first question is whether $A(s)$ is a bounded set. Set $\alpha_A(s) = \sup A(s)$.
In this regard our first result  is that if the associated graded ring 
$G_\m(A) = \bigoplus_{n \geq 0}\m^n/\m^{n+1}$ is a domain then $\alpha_A(s) \leq s e_0(A)$(here
$e_0(A)$ is the multiplicity of $A$). This motivates the question
\[
\text{Is} \ \alpha(A) = \limsup \frac{\alpha_A(s)}{s}  < \infty? 
\]
Of course if $G_\m(A)$ is a domain then $\alpha(A) \leq e_0(A)$.

If the residue field $A/\m$ is infinite and $\dim A \geq 2$ then it is not difficult to prove $\alpha(A) \geq 1$. We show
\begin{proposition}\label{ufd}
If $(A,\m)$ is a UFD  of dimension $d \geq 2$ and $G_\m(A)$ is a domain then $\alpha(A) = 1$. 
\end{proposition}
A natural question is what happens if $\alpha(A) = 1$. Although we are unable to solve this question, some 
elementary results led us to believe that the Chow groups will play  an important role in this endeavor.
Chow groups are important invariants  of a ring. See section \ref{sect-Chow} for definition of Chow groups.
Their utility in commutative algebra has been aptly demonstrated by Roberts, see \cite{Rob}.

Our first result is:
\begin{theorem}\label{dim2}
Let $A = \bigoplus_{n\geq 0} A_n$ be a finitely genrated graded domain over an algebraically closed field $k = A_0$. Note $A$ need not be standard graded. Assume $\dim A = 2$. Further assume
\begin{enumerate}[\rm (1)]
\item
There exists a Noether normalization $T = k[X_1, X_2]$  of $A$ with $\deg X_1 = \deg X_2 = m$.
\item
For every $\xi \in T_m$ with $\xi \neq 0$ we have $\sharp \Ass^{(1)}_A A/\xi A = 1$.
\end{enumerate}
If $R$ is a graded ring such that we have an inclusion of graded rings $T \subseteq R \subseteq A$ then  $A_1(R)$, the first Chow group of $R$, is a torsion group.
\end{theorem}

Although the hypotheses of Theorem \ref{dim2} look a bit contrived it is surprisingly effective for a large class of examples, for instance consider the following two:
\begin{enumerate}
\item
$A = k[X, Y, Z]/(X^n + Y^m + Z^n)$ where $k$ is algebraically closed, $m, n \geq 2$; $(m,n) = 1$ and $(mn)^{-1} \in k$. Grade $A$ by setting $\deg X = \deg Z = m$ and $\deg Y = n$. Note $T = k[X, Z]$ is a Noether normalization of $A$.
\item
$A = k[X, Y, Z]/(X^n + Y^m + X^aZ^b)$ where $k$ is algebraically closed, $m, n \geq 2$; $(m,n) = 1$, $a+b = n$ and $(mn)^{-1} \in k$. Also assume $a,b \geq 1$. Grade $A$ by setting $\deg X = \deg Z = m$ and $\deg Y = n$.  Note $T = k[X, Z]$ is a Noether normalization of $A$.
\end{enumerate}

It is very tempting to try and generalize Theorem \ref{dim2} to higher dimensions.
For this to work we have to work with a variant of Chow groups which is more suited for graded rings. 
We essentially consider only graded prime ideals and consider rational equivalence defined by going mod
homogeneous elements. See section \ref{sect-Chow} for this definition. Theorem \ref{dim-d} does not directly generalize. Our 
motivation was to ensure that generalizations of the above examples work in higher dimensions. 
\begin{theorem}\label{dim-d}
Let $A = \bigoplus_{n\geq 0} A_n$ be a finitely generated graded algebra over an algebraically closed field 
$k = A_0$. Note $A$ need not be standard graded. Assume $\dim A  = d \geq 2$ and that $A$ is \CM.. Further assume
\begin{enumerate}[\rm (1)]
\item
There exists a Noether normalization $T = k[X_1, \cdots, X_d]$  of $A$ with $\deg X_1 = \cdots = \deg X_d = m$.
\item
For every $T$-regular 
sequence $u_1,\ldots, u_{d-1}\in T_m$, 
we have \\ $\sharp \Ass^{(d-1)}_A A/(u_1,\ldots, u_{d-1}) A = 1$.
\end{enumerate}
If $R$ is a graded ring such that we have an inclusion of graded rings $T \subseteq R \subseteq A$
and $\grade(T_+, R) \geq d-1$ then  $A_{1}^*(R)$, the first  graded Chow group of $R$, is a torsion group.
\end{theorem}
We now extend the previous examples to give examples where our result holds:
\begin{enumerate}
\item
$A = k[X_1, X_2, \ldots, X_d, Y ]/( Y^m + \sum_{j = 1}^{d} X_j^n)$ where $k$ is algebraically closed, $m, n \geq 2$; $(m,n) = 1$ and $(mn)^{-1} \in k$. Grade $A$ by setting $\deg X_1 = \cdots =  \deg X_d = m$ and $\deg Y = n$. Note $T = k[X_1,\ldots, X_{d}]$ is a Noether normalization of $A$.
\item
$A = k[X, Y, Z, W_3, W_4, \ldots, W_d]/(X^n + Y^m + X^aZ^b + \sum_{j = 3}^{d} W_j^n)$ where $k$ is algebraically closed, $m, n \geq 2$; $(m,n) = 1$, $a+b = n$ and $(mn)^{-1} \in k$. Also assume $a,b \geq 1$. Grade $A$ by setting $\deg X = \deg Z = \deg W_3 = \cdots = \deg W_d = m$, $\deg Y = n$.  Note $T = k[X, Z, W_3, \cdots, W_d]$ is a Noether normalization of $A$.
\end{enumerate}

\begin{remark}
We note that even in dimension two;  previously the first Chow groups $A_1(A)$ was not known in the examples above.
\end{remark}

\begin{remark}
 Note that in dimension two the hypotheses of Theorem \ref{dim2} and Theorem \ref{dim-d}(in dimension 2) are different. It just 
 so happens that our examples satisfies the hypotheses of both the Theorems.
 \end{remark}

\section{Preliminaries}
In this section we  prove some preliminary results. We also show that certain polynomials are irreducible. 

We first prove:
\begin{lemma}\label{dom}
Let $(A,\m)$ be a Noetherian local ring of dimension $d$ with associated graded ring $G_\m(A)$ a domain. Let $x \in \m^s \setminus \m^{s+1}$. Then
$\Ass^{(1)}_A A/(x) \leq se_0(A)$.
\end{lemma}
\begin{proof}
We note that $A$ is a domain. Also note that $x^*$ the initial form of $x$ in $G_\m(A)$ is regular of order $s$. So $e_0(A/(x)) = se_0(A)$. We note that
$e_0(A/(x)) = \sum_{P}\lambda((A/(x))_P) e(A/P)$ where the sum is taken over all primes containing $(x)$ such that $\dim A/P = d-1$. The result follows.
\end{proof}
We now show
\begin{proposition}
 Let $(A,\m)$ be a Noetherian local ring of dimension $d \geq 2$ with $A/\m$ infinite. Assume $G_\m(A)$ is a domain.
 Then
 \begin{enumerate}[ \rm (1)]
  \item $\alpha(A) \geq 1$.
  \item If $A$ is also a UFD then $\alpha(A) = 1$. 
 \end{enumerate}
\end{proposition}
\begin{proof}
We note that $A$ is a domain.
 Let $x,y \in \m \setminus \m^2$  be part of system of parameters of $A$.
 
 (1) Let $\lambda$ be a unit in $A$. Set $t_\lambda = x + \lambda y$. It is elementary to see that for general $\lambda$ we have that
 $t_\lambda \in \m \setminus \m^2$. As $A$ is a domain we get $t_\lambda$ is regular.
 
 Fix $s \geq 2$. Let $u_s = t_{\lambda_1}\cdot t_{\lambda_2} \cdot \cdots t_{\lambda_s}$ where $\lambda_i$ are distinct and general. Then $u_s$ is regular.
 Also as $G_\m(A)$ is a domain
 we have that $u_s \in \m^s \setminus \m^{s+1}$.  If $P$ is a height one prime containing $u$ then
 $P$ contains $t_{\lambda_i}$ for some $i$. If 
 $P$ contains $t_{\lambda_i}$ and $t_{\lambda_j}$ for $i \neq j$ then it is easy to see that $x,y \in P$. This is a contradiction.
 Thus
$  \Ass_A^{(1)} A/(u_s)$ is equal to the disjoint union of $\Ass^{(1)}_A  A/(t_{\lambda_i})$. 
So $\alpha_A(s) \geq s$. Thus $\alpha(A)  \geq 1$.

(2). Let $u \in \m^s \setminus \m^{s+1}$. Then as $A$ is a UFD we have that $u = u_1^{r_1}\cdots u_t^{r_t}$ where $u_j$ are irreducible.
Clearly $\sum_j r_j \leq s$ In particular $t \leq s$. Also $  \Ass_A^{(1)} A/(u) = \{ (u_1), \cdots, (u_t) \}$. Thus
$\alpha_A(s) \leq s$.
It follows that  $\alpha(A)  \leq 1$. By (1) we get $\alpha(A) = 1$.
\end{proof}

I believe the following  results are
already known. However lack of a reference has forced me to include it here.

\begin{lemma}\label{d1}
Let $k$ be an algebraically closed field and consider the polynomial ring $S = k[X,Y]$. Let $m,n$ be integers such that $m, n \geq 2, (m,n)= 1$ and $(mn)^{-1} \in k$. Then 
\begin{enumerate}[\rm (1)]
\item
The polynomial $f = Y^m - X^n$ is irreducible in $S$.
\item 
For every $\alpha \neq 0$, $\alpha \in k$; the polynomial $Y^m  + \alpha X^n$ is irreducible in $S$.
\end{enumerate}
\end{lemma}
\begin{proof}
(1) Let $K = k(X)$. It suffices to prove that $f$ is irreducible in $K[Y]$. Let $\beta$ be a root of $f$. Then by \cite[Theorem 6.2, Chapter 6]{L} we get that the degree of the field extension $K(\beta)$ over $K$ is $d$ where $d|m$ and $\beta^d \in K$.
Let  $m = m_1d$. Note $\beta^d = X^{\frac{n}{m_1}} \in K$. As $(m_1, n)= 1$ this forces $m_1 =1$. So $d = m$. This implies that $f$ is irreducible in $K[Y]$.

(2) Let $\gamma \in k$ be such that $\gamma^n = -\alpha$. The result follows from a change of variables $Y \mapsto Y$ and $X \mapsto \gamma^{-1}Z$.
\end{proof}

We need the following:
\begin{proposition}\label{Gen}
Let $k$ be an algebraically closed field and
let  \\ $S = k[X, Y, Z, W_3,\ldots, W_d]$ with $d \geq 2$ (if $d = 2$ then $S = k[X,Y,Z]$). Let $m,n$ be integers such that $m, n \geq 2, (m,n)= 1$ and $(mn)^{-1} \in k$. 
Also $a, b \geq 1$ are integers  with $a+b = n$. Then the following polynomials are irreducible in $S$.
\begin{enumerate}[\rm(1)]
\item
$f = X^n + Y^m + Z^n + \sum_{j = 3}^{d}W_j^n$.
\item
$g = X^n + Y^m + X^aZ^b   + \sum_{j = 3}^{d}W_j^n$.
\end{enumerate}
\end{proposition}

To prove Proposition \ref{Gen},   it is convenient to prove the following:
\begin{lemma}\label{domain-reg}
Let $A = \bigoplus_{n \geq 0}A_n$ be a finitely generated graded algebra over a field $k = A_0$. (Note we do not assume $A$ is standard graded). Let $P = (u_1,\ldots, u_g)$ be a prime ideal such that $u_1,\ldots, u_g$ is an $A$-regular sequence. Furthermore $u_i$ are homogeneous for $i = 1,\cdots, g$. Then
\begin{enumerate}[\rm (1)]
\item
$A$ is a domain.
\item
$0 \subsetneq (u_1) \subsetneq (u_1, u_2) \subsetneq \cdots \subsetneq (u_1,\cdots, u_g)$ is a sequence of prime ideals in $A$.
\end{enumerate}
\end{lemma}
\begin{proof}
We prove the result by induction on $g$. We first consider the case $g = 1$.
Let $P = (u_1)$. Note $\height P = 1$ since $u_1$ is $A$-regular. Let $Q \subsetneq P$ be a minimal prime of $A$. Note $Q$ is a homogeneous ideal. Let $a \in Q$ be homogeneous. Then $a = r u_1$. As $u_1 \notin Q$ we get $r \in Q$. It follows that
$Q = u_1Q$. By graded Nakayama's Lemma we get $Q = 0$. So $A$ is a domain.

Let $g \geq 2$. We assume the result for $g-1$. Let $I = (u_1,\ldots, u_{g-1})$ and $B = A/I$. Then the ideal $\ov{P} = P/I = (\ov{u}_g)$ is prime in $B$. Furthermore $\ov{u}_g$ is $B$-regular. Thus by the $g = 1$ case we get that $B$ is a domain. So by induction we are done. 
\end{proof}
We now give
\begin{proof}[Proof of Proposition \ref{Gen}]
(1) Let $A = S/(f)$. Then $A$ is \CM. Furthermore $Z, W_3,\ldots,W_d$ is an $A$-regular sequence. Note 
\[
B = A/(Z, W_3,\ldots,W_d) \cong k[X,Y]/(X^n + Y^m).
\]
By Lemma \ref{d1} we get that $B$ is a domain. So by Lemma \ref{domain-reg} we get that $A$ is a domain. So $f$ is irreducible.

(2) This is similar to (1).
\end{proof}
\section{Chow groups}\label{sect-Chow}

We now give a brief discussion on Chow groups. We assume that for the rings under discussion there is a function $\cdim$ such that
\begin{enumerate}
 \item If $P \subseteq Q $ are adjacent primes then $\cdim A/Q = \cdim A/P - 1$.
 \item
 If $S$ is a multiplicatively closed set and $P$ is a prime ideal such that $P \cap S = \emptyset$ then $\cdim A/P = 
 \cdim A_S/P_S$.
 \item
 If $A$ is a affine algebra or a homorphic image of a regular local ring then $\cdim A = $ Krull dimension of $A$.
\end{enumerate}
For localizations of finitely generated algebra's over regular rings such a notion exists see \cite[p.\ 72]{Rob}.
We will only use it for affine algebra's. Furthermore the only localization we will consider is inverting a single element $x$.
Note that if $A$ is an affine algebra then $A_x = A[Y]/(xY - 1)$ is also an affine algebra.

We now define Chow groups for the rings  under consideration. Let $Z_i(A)$ be the free Abelian group with basis $[A/P]$ such that 
$P$ is a prime ideal in $A$ and $\cdim A/P = i$. If $M$ is a finitely generated module with $\cdim M \leq i$ then define the cycle
of dimension $i$ associated to $M$, denoted by $[M]_i$, to be the sum
\[
 \sum_{\cdim A/P = i}\text{length}(M_P)[A/P].
\]
Let $Q$ be a prime ideal with $\cdim A/Q = i+1$ and let $x$ be an element in $A$ not in $Q$ set $\divS(Q,x)$ to be 
the cycle $[(A/Q)/x (A/Q)]_i$. Rational equivalence  is the equivalence relation  on $Z_i(A)$  generated by setting $\divS(Q,x) = 0$
for all prime ideals $Q$ with $\cdim A/Q = i+1$ and for all $x \notin Q$. 
The Chow group of $A$ is the direct sum of the groups $A_i(A)$  where $A_i(A)$ is the group of cycles $Z_i(A)$ modulo rational equivalence.

\s \label{graded-gen} Now let $A = \bigoplus_{n \geq 0}A_n$ be a finitely generated grade domain over a field $k = A_0$. Note we are not assuming that
$A$ is standard graded. Let $d = \dim A$. 
The following result is well-known. We give a proof due to lack of a suitable reference.
\begin{proposition}\label{ht1}[with hypotheses as in \ref{graded-gen}]
 The group $A_{d-1}(A)$ is generated by $[A/P]$ where $P$ is a homogeneous prime of height one.
\end{proposition}
\begin{proof}
 Let $S$ be the group of all non-zero homogeneous elements in $A$. Then $A_S = L[t,t^{-1}]$ where $L$ is a field. By the localization
 sequence \cite[1.2.4]{Rob} we have an exact sequence
 \[
  0 \rt H \rt A_{d-1}(A) \rt A_{d-1}(A_S) \rt 0.
 \]
As $A_S$ is a UFD we have $A_{d-1}(A_S) = 0$. So $A_{d-1}(A) = H$. We note that 
$$H = <[A/P] \mid \height P = 1 \ \text{and} \ P \cap S \neq \emptyset >. $$ 
If $P$ is a height one prime in $A$ which contains a non-zero homogeneous element then $P$ is graded as $A$ is a domain. The result follows.
\end{proof}

We do not know if $A_i(A)$ is generated by homogeneous primes for $i \neq d-1$. 
To overcome this obstacle we define  \emph{homogeneous} Chow group of 
graded rings as follows. Let $A = \bigoplus_{n \in \ZZ}A_n$ be a Noetherian  $\ZZ$-graded ring which is a localization of a finitely 
generated algebra over a regular ring.
 Let $Z_i^*(A)$ be the free abelian group
generated by $[A/P]$ where $P$ is homogeneous and $\cdim A/P = i$. If $Q$ is homogeneous prime with $\cdim A/Q = i +1$ and $x$ is a 
homogeneous element not in $Q$ then note that all minimal primes of $(A/Q)/x(A/Q)$ are homogeneous. So $\divS(Q,x)\in Z_i^*(A)$.
Graded rational equivalence is the equivalence relation on $Z_i^*(A)$ generated by setting
$\divS(Q,x) = 0$ for all homogeneous primes $Q$ with $\cdim A/Q = i+1$ and for all homogeneous $x \notin Q$.
The graded Chow group of $A$ is the direct sum of $A_i^*(A)$ where  $A_i^*(A)$ is the group of cycles $Z_i^*(A)$ modulo graded rational 
equivalence.

The proof of the following result is similar to that of the ungraded case \cite[1.2.4]{Rob}.
\begin{proposition}
$A = \bigoplus_{n \in \ZZ}A_n$ be a Noetherian  $\ZZ$-graded ring which is a localization of a finitely generated algebra over a regular ring. Let $S$ be a multiplicatively closed set of homogeneous elements. Then for all $i \geq 0$ we have 
an exact sequence
\[
0 \rt H \rt A_i^*(A) \rt A_i^*(A_S) \rt 0;
\]
where $H = <[A/P] \mid \ \cdim A/P = i, P \ \text{homogeneous and} \ P \cap S \neq \emptyset  >$.
\end{proposition}

\section{Proof of Theorem \ref{dim2}}
In this section we give a proof of Theorem \ref{dim2}.

The following result is needed. I believe this is already known. However because of lack of a reference we give the proof here.
\begin{lemma}\label{inter}
Let $S$ be a Noetherian ring of dimension $d \geq 1$. Let $R = S[X]$ and $T = S[X, X^{-1}]$. Grade $R$ and $T$ by defining $\deg X = 1$. Then we have
\begin{enumerate}[\rm(1)]
\item
If $\m$ is a graded maximal ideal of $R$ then $\m = (\n , X)$ where $\n$ is a maximal ideal in $S$.
\item 
There is no homogeneous prime of height $d+1$ in  $T$.
\item
If $\m$ is a graded prime ideal in $T$ of height $d$ then $\m \cap S$ is a maximal ideal in $S$.
\end{enumerate}
\end{lemma}
\begin{proof}
(1) It suffices to show $X \in \m$. Set $Q = \m \cap S$. Suppose if possible $X \notin \m$. Let $f = \sum_{i = o}^{m}a_iX^i \in \m$. As $\m$ is homogeneous we get 
$a_iX^i \in \m$ for all $i$. As $X \notin \m$ we get that $a_i \in \m$ for all $i$. It follows that $\m = QR$. But $R/QR = S/Q[X]$ is not a field, a contradiction.

(2)  Suppose if possible there is a homogenous prime ideal $\n$ of height $d+1$ in $T$. 
Notice $\n \cap R$ is a graded prime ideal of height $d+1$. It follows that $\n \cap R$ is a graded maximal ideal of $R$.
By (1) we get $X \in \n \cap R$. It follows that $1 \in \n$, a contradiction. 

(3) Set $Q = \m \cap S$. Suppose if possible $Q$ is not a maximal ideal in $S$. Then 
$S/Q$ is not a field. So there exists $a \in S\setminus Q$ such that it's image $\ov{a}$ in $S/Q$ is not invertible.
We have an exact sequence
\[
0 \rt T/\m \xrightarrow{a} T/\m \rt T/(\m, a) \rt 0.
\]
We assert that $(\m , a) \neq T$. If this assertion is false then
$r + at = 1$ for some $r \in \m$ and $t \in T$. So we get $r_0 + at_0 = 1$. This implies that $\ov{a}$ is invertible in $S/Q$, a contradiction.
Note $\height (\m, a) = d+1$. Let $\n$ be any minimal prime of $(\m, a)$. Notice $\n$ is graded and $\height \n = d+1$.
This contradicts (2).
\end{proof}
We now give:
\begin{proof}[Proof of Theorem \ref{dim2}]
We first note that for any $\xi \in T_m$, $\xi \neq 0$ we have \\
$\sharp \Ass_R^{(1)} R/\xi R = 1$. To see this note that $A$ is finite as an $R$-module. Suppose if possible $P_1, P_2$ be two height $1$ primes lying over $\xi$. Let $Q_i$ for $i = 1,2$ be primes in $A$ lying above $P_i$. It is elementary to see that $\height Q_i = 1$ for $i = 1,2$. This gives  $\sharp \Ass_A^{(1)} A/\xi A \geq 2 $, a contradiction.

Notice $T_{X_2} = k[X_1/X_2][X_2, X_2^{-1}]$. Set $B = k[X_1/X_2]$. We have an isomorphism
\begin{align*}
B &\rt k[t] \\
X_1/X_2 &\mapsto t.
\end{align*} 
Notice $R_{X_2}$ is a finite extension of $T_{X_2}$. By the localization sequence we have
\[
0 \rt H \rt A_1^*(R) \rt A_1^*(R_{X_2}) \rt 0.
\]
Now
$$ H = \{ [R/P] \mid \dim R/P = 1, \ P \ \text{homogeneous and} \ X_2 \in P \}. $$
As $R$ is a catenary domain we get that $\height P = 1$ if $\dim R/P = 1$. 
As $\Ass^{(1)}_R R/X_2 R $ is a singleton set, say it is $\{ Q \}$ we get that $H = \ZZ [R/Q]$. Notice $\divS(0,X_2) = n[R/Q] $ for some $n \geq 1$. It follows that $H$ is a torsion group.

Let $P$ be a homogeneous prime of height one such that $X_2 \notin P$. Notice $\dim R/P = 1$. Also  note that $P_{X_2}$ is a homogeneous prime ideal of height one in $R_{X_2}$. So $L = P_{X_2}\cap T_{X_2}$ is a homogeneous prime ideal of height one in $T_{X_2}$. By Lemma \ref{inter}(2) we get that $L \cap B$ is a maximal ideal of $B$. As $k$ is algebraically closed we get $L \cap B = (X_1/X_2 - c)$ for some $c \in k$. It follows that $\xi = X_1 - cX_2 \in P$. By our assumptions we get that
$\Ass_A^{(1)} R/\xi R = \{ P \}$. Notice $\divS(0,X_2) = n[R/P] $ for some $n \geq 1$. So $n [R/P] = 0$ in $A_1^*(R)$. Thus $A_1^*(R)$ is a torsion group.

The natural map $A_1^*(R) \rt A_1(R)$ is surjective by Proposition \ref{ht1}. The result follows.
\end{proof}

\section{Proof of Theorem \ref{dim-d} }
In this section we prove Theorem \ref{dim-d}. We need a few preparatory results. The first one is:
\begin{lemma}\label{H}
Let $R = \bigoplus_{n \geq 0}R_n$ be a finitely generated graded  algebra  over a field $k = R_0$. Assume $\dim R = d$.
Let $z $ be a non-zero homogeneous regular  element in $R$. Set
$$ H = < [R/P] \mid  \dim R/P = 1, \text{$P$ is homogeneous and} \ z \in P >;$$
considered as a subgroup of $A_{1}^*(R)$. Set $S = R/(z)$. For a prime $Q$ in $S$ with $\dim S/Q = 1$ let $P_Q$ be the prime in $R$ with $P_Q/(z) = Q$. Then the map
\begin{align*}
A_{1}^*(S) &\rt H \\
[S/Q] &\mapsto [R/P_Q] 
\end{align*}
is well-defined and surjective. In particular if $A_1^*(S)$ is a torsion group 
 then so is $H$.
\end{lemma}
\begin{proof}
Note that we have a natural map
\begin{align*}
\eta \colon Z_1^*(S) &\rt H \\
[S/Q] &\mapsto [R/P_Q] 
\end{align*}
which is well-defined and surjective. Let $L$ be a prime ideal in $S$ with $\dim S/L = 2$ and let $\xi \in B$ be homogeneous
such that $\xi \notin L$. Let $P_L$ be prime ideal in $R$ with $P_L/(z) = L$ and let $\xi^\prime $ be any homogeneous
element in $R$ such that it's image in $S$ is $\xi$.

Set $M = S/(L, \xi) = R/(P_L , \xi^\prime)$. If $Q$ is a minimal prime of $M$ as an $S$-module then note that 
$P_Q$ is a minimal prime of $M$ as a $R$-module. 
Also note that $M_Q = M_{P_Q}$. Thus $\eta(\divS(\xi, L)) = \divS(\xi^\prime, P_L)$. 
So $\eta$ factors through a map $\ov{\eta} \colon A_1^*(S) \rt H $ which is the map we needed to show is well-defined. Also clearly $\ov{\eta}$ is surjective since $\eta$ is surjective. 
\end{proof}
The second preparatory result is the following:
\begin{proposition}\label{sec-prep}
 Let $A = \bigoplus_{n \geq 0}A_n $ be a finitely generated graded algebra over a field $k = A_0$. Note we do not assume that $A$
 is standard graded. Assume $d = \dim A \geq 2$ and that $A$ has a Noether normalization $T = k[X_1,\ldots, X_d]$ where $\deg X_1 = \cdots = \deg X_d = m$.
 Also assume that $A$ is \CM. 
 Let $R$ be any graded ring such that we have inclusions of graded rings $T \subseteq R \subseteq A$.
 We then have the following:
 \begin{enumerate}[\rm (1)]
  \item $R$ is equidimensional, i.e., $\dim R/P = \dim R$ for all minimal prime $P$ of $R$.
  \item
  Let $ \bu = u_1,\ldots, u_g$ be a homogeneous regular sequence in $T$. Let $P$ be any minimal prime in $R$ of 
  $\bu R$. Then
  \begin{enumerate}[\rm (a)]
   \item 
   $\height P = g$
   \item
   If $\sharp \Ass_A^{(g)} A/(\bu)A = 1$ then $\sharp \Ass_R^{(g)} R/(\bu)R = 1$.
  \end{enumerate}
\item
If $P$ is a homogenous prime ideal in $R$ then $\dim R/P = \dim R - \height P$.
  \end{enumerate}
\end{proposition}
\begin{proof}
 (1) Let $Q$ be any prime in $A$ lying above $P$. Then notice $\height Q = 0$.
 As $A$ is \CM \ we have $\dim A = \dim A/Q$. As $A/Q$ is integral over $R/P$ we have $\dim R/P = \dim A/Q$. Finally note that
 $\dim R = \dim A$. So the result follows.
 
 (2)(a) By Krull's principal ideal Theorem we have $\height P \leq g$. Let $Q$ be a prime in $A$ lying over $P$. Notice
 as $A$ is \CM \ it is free over $T$. So $\bu$ is an $A$-regular sequence. In particular $\height Q \geq g$. So $\height P \geq g$.
 Thus $\height P = g$.

 (2)(b) By 2(a) it follows that all the primes considered are graded.  Suppose if possible $P_1,P_2 \in \Ass_R^{(g)} R/(\bu R)$. Let $Q_1, Q_2$ be primes in $A$ lying above $P_1$, $P_2$ respectively.
 It is easy to see that $\height Q_1 = \height Q_2 = g$. This gives $\sharp \Ass_A^{(g)} A/(\bu)A  \geq 2$, a contradiction. 
 
 (3) We note that as $R$ is an affine algebra it is catenary. The result follows by an argument similar to \cite[Lemma 2, p.\ 250]{Mat}.
 \end{proof}

We now give
\begin{proof}[Proof of Theorem \ref{dim-d}]
We prove our result by induction on $d = \dim A$.
 We first prove the result when $d = 2$.
 Notice   $T_{X_2} = k[X_1/X_2][X_2, X_2^{-1}]$. Set $B = k[X_1/X_2]$. We have an isomorphism
\begin{align*}
B &\rt k[t] \\
X_1/X_2 &\mapsto t.
\end{align*} 
Notice $R_{X_2}$ is a finite extension of $T_{X_2}$. By the localization sequence we have
\[
0 \rt H \rt A_1^*(R) \rt A_1^*(R_{X_2}) \rt 0.
\]
Now
$$ H = \{ [R/P] \mid \dim R/P = 1, \ P \ \text{homogeneous and} \ X_2 \in P \}. $$
By \ref{sec-prep} we get  $\height P = 1$ if $\dim R/P = 1$. 
By \ref{sec-prep},  $\Ass^{(1)}_R R/X_2 R $ is a singleton set, say it is $\{ Q \}$ we get that $H = \ZZ [R/Q]$. 
Let $Q^\prime$ be a minimal prime of $R$ with $Q^\prime \subseteq Q$. We note that $Q$ is the unique prime of height one containing $(Q^\prime, X_2)$.
Therefore $\divS(Q^\prime,X_2) = n[R/Q] $ for some $n \geq 1$. It follows that $H$ is a torsion group.

Let $P$ be a homogeneous prime of height one such that $X_2 \notin P$. Notice $\dim R/P = 1$. 
Also  note that $P_{X_2}$ is a homogeneous prime ideal of height one in $R_{X_2}$. 
So $L = P_{X_2}\cap T_{X_2}$ is a homogeneous prime ideal of height one in $T_{X_2}$. 
By Lemma \ref{inter}(3) we get that $L \cap B$ is a maximal ideal of $B$. As $k$ is algebraically closed we get 
$L \cap B = (X_1/X_2 - c)$ for some $c \in k$. It follows that $\xi = X_1 - cX_2 \in P$. By \ref{sec-prep} we get that
$\Ass_A^{(1)} R/\xi R = \{ P \}$. Let $Q$ be a minimal prime of $R$ with $Q \subseteq P$. We note that $P$ is the
unique prime of height one containing $(Q, \xi)$.
Therefore $\divS(Q,\xi) = n[R/P] $ for some $n \geq 1$.  So $n [R/P] = 0$ in $A_1^*(R)$. 
Thus $A_1^*(R)$ is a torsion group.

Now assume $d = \dim A \geq 3$ and the result has been proved for $d-1$. Notice $\grade(T_+, R/T ) \geq d-1$ and 
$\grade(T_+, A/R ) \geq d-2$.
As $k$ is algebraically closed it is in particular an infinite field. So there exists $\xi \in T_m$ such that $\xi $ is $R/T \oplus A/R$
regular. After a change of variables we may assume that $T = k[\xi = \xi_1, \xi_2,\ldots, \xi_d]$ where $\deg \xi_i = m$ for all
$i$. Set $\ov{T} = T/(\xi), \ov{R} = R/\xi R$ and $\ov{A} = A/\xi A$. Note as $\xi$ is $R/T\oplus A/R$-regular, we have an inclusion
of graded rings $\ov{T} \subseteq \ov{R} \subseteq \ov{A}$. Notice
\begin{enumerate}
 \item $\ov{T}$ is a Noether normalization of $\ov{A}$.
 \item $\ov{A}$ is \CM.
 \item $\grade(\ov{T}_+, \ov{R}) \geq d -2$.
 \item $\ov{T} = k[\xi_2,\ldots, \xi_d]$ where $\deg \xi_i = m$ for all $i$.
 \item 
 If $v_1,\ldots, v_{d-2} \in \ov{T}_m$ is a regular sequence then choosing a preimage $u_i$ of $v_i$ in $T_\m$ we get that
 $\xi, u_1,\ldots, u_{d-2} \in T_m$ is a $T$-regular sequence. So by our hypotheses 
 $$ \sharp \Ass_A^{(d-1)}A/(\xi, u_1,\ldots, u_{d-2})A = 1.$$
 Notice 
 $$A/(\xi, u_1,\ldots, u_{d-2})A = \ov{A}/(v_1,\ldots, v_{d-2})\ov{A}.$$
It follows that 
$$ \sharp \Ass_{\ov{A}}^{(d-2)} \ov{A}/( v_1,\ldots, v_{d-2})\ov{A} = 1.$$
 \end{enumerate}
 So by our inductive hypothesis we obtain $A_{1}^*(\ov{R})$ is a torsion group.

By the localization sequence we have 
\[
  0 \rt H \rt A_{d-1}^*(R) \rt A_{d-1}^* R_\xi \rt 0;
\]
where 
$$ H = < [R/P] \mid \dim R/P = 1 \ \text{$P$ homogeneous and} \ \xi \in P >. $$
By \ref{H} we have a surjective homomorphism $A_1^*(\ov{R}) \rt H$. It follows that $H$ is a torsion group.

Let $ \dim R/P = 1$ and assume $\xi \notin P$.
By \ref{sec-prep} we get that if $P$ is homogeneous then  $\dim R/P = 1$ if and only if $\height P = d-1$.
 Thus $PR_\xi$ is a graded prime ideal of $R_\xi$ with height  = $d-1$.

 $R_\xi$ is a finite extension of $T_\xi$. Notice $T_\xi = k[\xi_2/\xi,\ldots,\xi_d/\xi][\xi, \xi^{-1}]$. Set 
$B = k[\xi_2/\xi,\ldots,\xi_d/\xi]$. We note that $B \cong k[Y_2,\ldots, Y_d]$ under the mapping $\xi_i/\xi \mapsto Y_i$.
$L = PR_\xi \cap T_\xi$ is a graded prime ideal of height $d-1$. It follows that $L \cap B$ is a maximal ideal in $B$, see \ref{inter}(3).
Thus $L \cap B = (\xi_2/\xi -c_2, \ldots, \xi_d/\xi - c_d)$ for some $c_i \in k$. It follows that
$$ \xi_2 - c_2\xi,\ldots, \xi_d - c_d\xi \in P. $$
Notice $\xi_2 - c_2\xi, \ldots, \xi_d - c_d$ is a $T$-regular sequence. It follows from our hypothesis that
\[
 \Ass^{(d-1)}_R R/(\xi_2 -c_2\xi,\cdots, \xi_d - c_d\xi) = \{ P \}.
\]
By \ref{sec-prep} we have that all minimal primes of $(\xi_2-c_2 \xi,\cdots, \xi_{d-1}-c_{d-1}\xi)R$ will have height $d-2$. Furthermore it is clear that
at least one of them, say $Q$, is contained in $P$. It is elementary to see that
\[
 \Ass^{(d-1)}_R R/(Q, \xi_d - c_d \xi)R  = \{ P \}.
\]
So $\divS(Q, \xi_d - c_d \xi) = n[R/P]$ for some $n \geq 1$. Thus in $A_1^*(R)$ we have $n[R/P] = 0$.
It follows that 
$A_1^*(R)$ is a torsion group.
\end{proof}

\section{examples}
In this example we explicitly give examples of \CM \ domains which satisfy the hypotheses of $A$ in  Theorem \ref{dim-d}. They also satisfy the hypotheses of Theorem \ref{dim2} when $\dim A = 2$.
\begin{example}\label{ex-1}
$A = k[X_1, X_2, \ldots, X_d, Y ]/( Y^m + \sum_{j = 1}^{d} X_j^n)$ where $k$ is algebraically closed, $m, n \geq 2$; $(m,n) = 1$ and $(mn)^{-1} \in k$. Grade $A$ by setting $\deg X_1 = \cdots =  \deg X_d = m$ and $\deg Y = n$. Note $T = k[X_1,\ldots, X_{d}]$ is a Noether normalization of $A$. Also note that $A$ is \CM.
By \ref{Gen} we also get that $A$ is a domain.

Let $\bu = u_1,\ldots, u_{d-1} \in T_m$ be a $T$-regular sequence. Then there exists
$v \in T_m$ such that $(\bu, v) = (X_1,\ldots, X_d)$.
Assume $$X_i = \sum_{j = 1}^{d-1}a_{ij}u_j + \beta_i v.$$
Then
\begin{align*}
A/(\bu)A &= \frac{k[\bu, v, Y]}{(Y^m + \sum_{i = 1}^{d} \left( \sum_{j = 1}^{d-1}a_{ij}u_j + \beta_i v \right)^n, \bu ) } \\
   &\cong \frac{k[ v, Y]}{(Y^m + \sum_{i=1}^{d} \beta_i^nv^n)} \\
   &\cong \frac{k[ v, Y]}{(Y^m + \left(\sum_{i=1}^{d} \beta_i^n \right)v^n)}.
\end{align*} 
Set $\beta = \sum_{i=1}^{d} \beta_i^n$. We have to consider two cases.

Case 1: $\beta = 0$. \\
Then $A/(\bu)A \cong k[v,Y]/(Y^m)$. Clearly this has only one minimal prime.

Case 2: $\beta \neq 0$. \\
Then $A/(\bu)A \cong k[v,Y]/(Y^m + \beta v^n)$ is a domain, see \ref{d1}. Clearly this has only one minimal prime. 
\end{example}

The second class of  examples of \CM \ domains satisfying our hypotheses is the following:
\begin{example}\label{ex2}
$A = k[X, Y, Z, W_3, W_4, \ldots, W_d]/(X^n + Y^m + X^aZ^b + \sum_{j = 3}^{d} W_j^n)$ where $k$ is algebraically closed, $m, n \geq 2$; $(m,n) = 1$, $a+b = n$ and $(mn)^{-1} \in k$. Also assume $a,b \geq 1$. Grade $A$ by setting $\deg X = \deg Z = \deg W_3 = \cdots = \deg W_d = m$, $\deg Y = n$.  Note $T = k[X, Z, W_3, \cdots, W_d]$ is a Noether normalization of $A$.  Also note that $A$ is \CM.
By \ref{Gen} we also get that $A$ is a domain.

Let $\bu = u_1,\ldots, u_{d-1} \in T_m$ be a $T$-regular sequence. Then there exists
$v \in T_m$ such that $(\bu, v) = (X, Z, W_3, \cdots, W_d)$.

Let
\begin{align*}
X &= \sum_{i = 1}^{d-1}a_i u_i  + bv, \\
Z  &= \sum_{i = 1}^{d-1}c_i u_i  + dv,\\
W_j  &= \sum_{i=1}^{d-1}e_{ij}u_i + f_jv; \ \text{for} \ j = 3,\ldots, d.
\end{align*}
Then by an argument similar to the previous example we get
\begin{align*}
A/(\bu)A  &\cong \frac{k[ v, Y]}{(Y^m +  b^n v^n + b^ad^b v^n + \sum_{j=3}^{d} f_j^nv^n)} \\
   &\cong \frac{k[ v, Y]}{(Y^m + \left( b^n  + b^ad^b +  \sum_{j=3}^{d} f_j^n  \right)v^n)}.
\end{align*} 
Set $\beta = b^n  + b^ad^b +  \sum_{j=3}^{d} f_j^n $.
We have to consider two cases.

Case 1: $\beta = 0$. \\
Then $A/(\bu)A \cong k[v,Y]/(Y^m)$. Clearly this has only one minimal prime.

Case 2: $\beta \neq 0$. \\
Then $A/(\bu)A \cong k[v,Y]/(Y^m + \beta v^n)$ is a domain, see \ref{d1}. Clearly this has only one minimal prime. 
\end{example}

\end{document}